\documentclass{amsart}%[a4paper]
\usepackage{graphicx}
\usepackage{amsmath}
\usepackage{amssymb}
\usepackage[all]{xypic}

\newcommand{\Hom}{\operatorname{Hom}\nolimits}

\renewcommand{\Im}{\operatorname{Im}\nolimits}

\newcommand{\Ker}{\operatorname{Ker}\nolimits}

\newcommand{\pd}{\operatorname{pd}\nolimits}

\newcommand{\Ktac}{\operatorname{{\bf K}_{tac}}\nolimits}
\newcommand{\Db}{\operatorname{{\bf D}^{b}}\nolimits}
\newcommand{\Dbsg}{\operatorname{{\bf D}^{b}_{sg}}\nolimits}
\newcommand{\Id}{\operatorname{Id}\nolimits}
\renewcommand{\H}{\operatorname{H}\nolimits}

\newcommand{\rank}{\operatorname{rank}\nolimits}

\newcommand{\xra}{\xrightarrow}

\newcommand{\cone}{\operatorname{cone}\nolimits}
\newcommand{\trcone}{\operatorname{trcone}\nolimits}

\newtheorem{theorem}{Theorem}[section]

\newtheorem{corollary}[theorem]{Corollary}

\newtheorem{lemma}[theorem]{Lemma}
\newtheorem{proposition}[theorem]{Proposition}
\theoremstyle{definition}
\newtheorem*{definition}{Definition}
\newtheorem{chunk}[theorem]{}
\theoremstyle{definition}
\newtheorem{question}[theorem]{Question}

\theoremstyle{definition}

\theoremstyle{definition}
\newtheorem{example}[theorem]{Example}
\theoremstyle{definition}

\theoremstyle{definition}

\theoremstyle{remark}

\theoremstyle{definition}

\theoremstyle{definition}

\begin{document}

\title{Totally acyclic approximations}
\author{Petter A. Bergh, David A.\ Jorgensen, and W. Frank Moore}

\address{Petter Andreas Bergh \\ Institutt for matematiske fag \\
  NTNU \\ N-7491 Trondheim \\ Norway}
\email{bergh@math.ntnu.no}

\address{David A.\ Jorgensen \\ Department of mathematics \\ University
of Texas at Arlington \\ Arlington \\ TX 76019 \\ USA}
\email{djorgens@uta.edu}

\address{W. Frank Moore \\ Department of mathematics \\ University
of Texas at Arlington \\ Arlington \\ TX 76019 \\ USA}
\email{djorgens@uta.edu}

\date{\today}



\begin{abstract}
Let $R$ be a commutative local ring. We study the subcategory of the homotopy category of $R$-complexes consisting of the totally acyclic  $R$-complexes.  In particular, in the context where $Q\to R$ is a surjective local ring homomorphism such that $R$ has finite projective dimension over $Q$, we define an adjoint pair of functors between the homotopy category of totally acyclic $R$-complexes and that of $Q$-complexes, which are analogous to the classical adjoint pair between the module categories of $R$ and $Q$.  We give detailed proofs of the adjunction in terms of the unit and counit.  As a consequence, one obtains a precise notion of approximations of totally acyclic $R$-complexes by totally acyclic $Q$-complexes.
\end{abstract}

\maketitle

\section*{Introduction}\label{intro}

This paper was motivated by the desire to approximate in a meaningful way totally acyclic complexes over a commutative local ring $R$ by simpler totally acyclic complexes, possibly even periodic ones.   The subcategory $\Ktac(R)$ of the homotopy category of $R$-complexes consisting of totally acyclic $R$-complexes is a thick subcategory, and therefore is also a triangulated category.  To facilitate an approximation, we assume that there exists a surjective local ring homomorphism $\varphi:Q \to R$ such that $Q$ is Gorenstein and $R$ has finite projective dimension as a $Q$-module.  Our approximation is then achieved through an adjoint pair of triangle functors
\[
\xymatrixrowsep{2pc}
\xymatrixcolsep{3pc}
\xymatrix{
\Ktac(Q) \ar@{->}[r]<0.5 ex>^{S_\varphi} & \Ktac(R) \ar@{->}[l]<0.5 ex>^{T_\varphi}
}
\]
which are consonant with the classical adjoint pair of functors between the module categories of 
$Q$ and $R$. Indeed, as is the case in the classical setting, $S_\varphi$ is simply the base change functor.  However, as nontrivial totally acyclic $R$-complexes are never totally acyclic $Q$-complexes when $R\ne Q$, obtaining the right adjoint $T_\varphi$ requires a modification of the forgetful functor.  Its construction is given in Section 2, and we prove in detail that it is a triangle functor.  We remark also that the condition of $\pd_QR<\infty$ is key to the existence of $T_\varphi$. In Section 3 we prove that $S_\varphi$ and $T_\varphi$ form an adjoint pair via the unit and counit natural transformations. In Section 4 we recall the precise notion of approximation, in terms of the counit of the adjunction. The approximations are especially meaningful in Remark \ref{period2}, where totally acyclic complexes over a quotient of a hypersurface ring are approximated by period 2 complexes.  Our interest in this project was also motivated by recent work in \cite{Krause}, \cite{Neeman} on approximations and adjoints in related categories.

To put the functors $S=S_\varphi$ and $T=T_\varphi$ in further perspective, Let $\Db(R)$ denote the bounded derived category
of $R$ and $\Db(Q)$ that of $Q$.  The classical adjoint pair of the base change functor and the forgetful functor for the module categories extend to an adjoint pair of the bounded derived categories.  Since the projective dimension of $R$ over $Q$ is finite, a perfect complex in $\Db(R)$ under the forgetful functor is a perfect complex in $\Db(Q)$.  Thus the adjoint pair extends to the verdier quotients of the bounded derived category modulo the corresponding thick subcategories of perfect complexes, in other words the singularity categories $\Dbsg(Q)$ and $\Dbsg(R)$.  This explains the right pair of vertical arrows in the diagram below.  The top horizontal arrow is the equivalence of categories proved in \cite{Buchweitz}; the functor is defined by hard truncation of a totally acyclic complex to the right of zero.  It is proved in  \cite{BerghJorgensenOppermann} that the same functor is fully faithfully in general, which explains the bottom arrow. 
\[
\xymatrixrowsep{2pc}
\xymatrixcolsep{3pc}
\xymatrix{
\Ktac(Q) \ar@{<-}[d]<0.5 ex>^T \ar@{->}[r]^{\cong} & \Dbsg(Q) \ar@{<-}[d]<0.5 ex>^\tau \\
\Ktac(R) \ar@{<-}[u]<0.5 ex>^S \ar@{^{(}->}[r] & \Dbsg(R) \ar@{<-}[u]<0.5 ex>^\sigma
}
\]
The adjoint pair $S$ and $T$ of this paper is the corresponding adjoint pair making the square commute.

\section{Preliminaries}

Let $R$ be a ring.  By an \emph{$R$-complex} $C$ we mean a sequence of left $R$-module 
homomrphisms
\[
C:  \cdots \to C_{n+1} \xrightarrow{\partial^C_{n+1}} C_n \xrightarrow{\partial^C_n} C_{n-1} \to \cdots
\]
graded homologically, where each $C_n$ is a left $R$-module, and $n$ is its \emph{homological degree}.  Without prior stipulation, an arbitrary $R$-module is assumed also to be a complex concentrated in homological degree 0.  

Recall that a \emph{semi-free} $R$-complex $F$ is one whose underlying graded
$R$-module $F^\natural$ has graded basis $E$ which can be written as a disjoint union 
$E=\bigsqcup_{n\ge 0}E^n$ such that $\partial^F(E^n)\subseteq R\langle \bigcup_{i=1}^{n-1}E^i\rangle$ 
for every $n\ge 1$.

\begin{definition} A \emph{semi-free resolution} of an $R$-complex $C$ is a quasi-isomorphism $F \to C$ of complexes where $F$ is a semi-free $R$-complex.
\end{definition}

We now state some important facts, and simply give a reference when the proofs are available in the literature.

\begin{chunk}\label{semifreeresolution} \cite[5.1.7, 5.1.13]{CFH} Every $R$-complex  $C$ has a semi-free resolution $\pi : F\to C$ with
$F_n = 0$ for all $n < \inf\{i \mid  \H_i(C)\ne 0\}$. Moreover, $\pi$ can be chosen to be surjective.  If $R$ is left Noetherian and $C$ is a bounded below complex of finitely generated $R$-modules, then $\pi$ can even be chosen to be surjective with $F_n$ finitely generated for all $n$, and
$F_n=0$ for all $n<\inf\{i\mid C_i\ne 0\}$.
\end{chunk}

The next result is the so-called `Comparison Theorem.'

\begin{chunk}\label{comparison}\cite[5.2.9]{CFH}  Let  $\pi':F'\to C'$ be a quasiisomorphism of $R$-complexes, and $\alpha:F\to C'$ a morphism of $R$-complexes with $F$ semi-free.  Then there exists a morphism of complexes $\mu:F \to F'$ such that $\pi'\mu\sim \alpha$.  If $\pi'$ is surjective, then
$\mu$ can be chosen such that $\pi'\mu=\alpha$. Moreover, $\mu$ is homotopic to any other morphism of $R$-complexes $\mu'$ with
$\pi'\mu'=\alpha$.
\end{chunk}

From this point on we assume that $R$ is a commutative Noetherian ring. 

\begin{definition} Recall from \cite{AvramovMartsinkovsky} that an $R$-complex $C$ of finitely generated free modules is called \emph{totally acyclic} if
\[
\H(C)=0=\H(\Hom_R(C,R)).
\]
We denote by $\Ktac(R)$ the homotopy category of totally acyclic $R$-complexes; the objects in
$\Ktac(R)$ are the totally acyclic $R$-complexes, and the morphisms are homotopy equivalence classes of morphisms of $R$-complexes.  For a morphism $f:C\to C'$ of $R$-complexes we write 
$[f]$ for its homotopy equivalence class.  Thus for two morphisms $f,g:C\to C'$ of $R$-complexes, one has $f\sim g$ if and only if $[f]=[g]$.
\end{definition}

The following facts we use often in the rest of the paper. Given an $R$-complex $C$, from now on we write $C^*$ for the $R$-complex $\Hom_R(C,R)$.

\begin{chunk}\label{extendhomotopy} Suppose that $s$ is an integer, $C, C'\in\Ktac(R)$ and $f:C\to C'$ is a morphism of complexes.  If there exist maps
$\sigma_n:C_n\to C'_{n+1}$ satisfying $f_n=\sigma_{n-1}\partial_n^C+\partial_{n+1}^{C'}\sigma_n$ for all $n > s$, then the $\sigma_n$ can be extended to a homotopy showing that $f\sim 0$.
\end{chunk}

\begin{proof} By induction it suffices to show there exists a map $\sigma_s:C_s\to C'_{s+1}$
such that $f_{s+1}=\sigma_s\partial_{s+1}^C+\partial_{s+2}^{C'}\sigma_{s+1}$.

Note that since $C$ is a totally acyclic complex and $C'_{s+1}$ is free, the complex 
$\Hom_R(C,C'_{s+1})$ is exact.  We have 
\begin{align*}
\partial_{s+1}^{\Hom_R(C,C'_{s+1})}(f_{s+1}-\partial_{s+2}^{C'}\sigma_{s+1})&=
\left(f_{s+1}-\partial_{s+2}^{C'}\sigma_{s+1}\right)\partial^C_{s+2}\\
&=f_{s+1}\partial^C_{s+2}-\partial^{C'}_{s+2}\left(f_{s+2}-\partial^{C'}_{s+3}\sigma_{s+2}\right)\\
&=\left(f_{s+1}\partial^C_{s+2}-\partial^{C'}_{s+2}f_{s+2}\right)-\partial_{s+3}^{C'}\partial^{C'}_{s+3}\sigma_{s+2}\\
&=0
\end{align*}
Thus $f_{s+1}-\partial_{s+2}^{C'}\sigma_{s+1}\in\Ker\partial_{s+1}^{\Hom_R(C,C'_{s+1})}
=\Im\partial_{s}^{\Hom_R(C,C'_{s+1})}$.  Therefore there exists a map
$\sigma_s:C_s\to C'_{s+1}$ such that 
$\partial_{s}^{\Hom_R(C,C'_{s+1})}(\sigma_s)=f_{s+1}-\partial_{s+2}^{C'}\sigma_{s+1}$, in other words $f_{s+1}=\sigma_s\partial_{s+1}^C+\partial_{s+2}^{C'}\sigma_{s+1}$.
\end{proof}

\begin{chunk}\label{extendmorphism} Suppose that $s$ is an integer and $C, C'\in\Ktac(R)$.  If there exist maps
$f_n:C_n\to C'_n$ satisfying $f_{n-1}\partial_n^C=\partial_n^{C'}f_n$ for all $n>s$, then the $f_n$ can be extended to a morphism of complexes $f:C\to C'$.  Any two such extensions are homotopic, that is, if $g:C \to C'$ is a morphism of complexes such that there exist maps $\sigma_n:C_n\to C'_{n+1}$
with $f_n-g_n=\sigma_{n-1}\partial_n^C+\partial_n^{C'}\sigma_n$ for all $n\ge s$, then $f\sim g$.
\end{chunk}

\begin{proof} The second statement follows immediately from \ref{extendhomotopy}.  The proof of the first statement is essentially that of \ref{extendhomotopy}, and is left to the reader.
\end{proof}

\begin{definition} A \emph{complete resolution} (see, for example, \cite{AvramovMartsinkovsky}) of an $R$-complex $C$ is a diagram of morphisms of $R$-complexes
\[
U\xrightarrow{\rho} F \xrightarrow{\pi} C
\]
such that $U\in\Ktac(R)$, $F$ is a semi-free resolution of $C$, and $\rho_n$ is bijective for all $n \gg 0$.  We will often abuse terminology and call $U$ a complete resolution of $C$.
\end{definition}

The following is a slightly generalized version of \cite[5.3]{AvramovMartsinkovsky}.  The proof
is essentially the same as that in ibid.

\begin{chunk}\label{well-defined} Complete resolutions are well-defined up to homotopy equivalence.  Specifically, if
$U\xrightarrow{\rho} F \xrightarrow{\pi} C$ and $U'\xrightarrow{\rho'} F' \xrightarrow{\pi'} C'$ are
complete resolutions of $R$-complexes $C$ and $C'$, and $\mu:C\to C'$ is a morphism of complexes, then by the comparison theorem \ref{comparison} there exists a unique-up-to-homotopy morphism $\overline\mu$ making the right-hand square of the diagram
\[
\xymatrixrowsep{2pc}
\xymatrixcolsep{3pc}
\xymatrix{
U\ar@{->}[r]\ar@{.>}[d]^{\widehat\mu} & F \ar@{.>}[d]^{\overline\mu} \ar@{->}[r] & C \ar@{->}[d]^{\mu} \\
U'\ar@{->}[r] & F' \ar@{->}[r] & C'
}
\]
commute up to homotopy, and for each choice of $\overline\mu$ there exists a 
unique-up-to-homotopy  morphism  $\widehat\mu$ making the left-hand square commute up to homotopy.  If $\mu=\Id_C$, then $\overline\mu$ and $\widehat\mu$ are homotopy equivalences.
\end{chunk}

\begin{chunk}\label{dual isomorphism}  Suppose that $Q$ and $R$ are commutative Noetherian rings, and $\varphi:Q\to R$ is a surjective ring homomorphism. Let $C$ be a complex of finitely generated free $Q$-modules.  Then 
\[
\Hom_Q(C,Q)\otimes_QR \text{ and } \Hom_R(C\otimes_QR,R) 
\]
are isomorphic as $R$-complexes. 
\end{chunk}

\begin{proof} By Hom-tensor adjunction for complexes, and after making the canonical identification of $\Hom_R(R,R)$ with $R$, we immediately get that $\Hom_R(C\otimes_Q R,R)$ and 
$\Hom_Q(C,R)$ are isomorphic.  Therefore it suffices to prove that $\Hom_Q(C,Q)\otimes_QR$ and
$\Hom_Q(C,R)$ are isomorphic.

Define maps $\alpha_n:\Hom_Q(C_{-n},Q)\otimes_QR\to\Hom_Q(C_{-n},R)$ by 
$\alpha_n(f\otimes r)(x)=\varphi(f(x))r$, and 
$\beta_n:\Hom_Q(C_{-n},R)\to\Hom_Q(C_{-n},Q)\otimes_QR$ by
$\beta_n(g)=g'\otimes_Q 1_R$ where $\varphi g'=g$.  Thus
$\alpha_n(\beta_n(g))(x)=\alpha_n(g'\otimes 1_R)(x)=\varphi(g'(x))=g(x)$, and so
$\alpha_n\beta_n=\Id_{\Hom_Q(C_{-n},R)}$.  Also, 
$\beta_n(\alpha_n(f\otimes r))=g'\otimes 1_R$ where $\varphi(g'(x))=\varphi(f(x))r$,
and so $\beta_n\alpha_n=\Id_{\Hom_Q(C_{-n},Q)\otimes_QR}$.  Finally one just
needs to check that $\alpha_n$ and $\beta_n$ commute with the differentials, and this is left to the reader.
\end{proof}

%A complex $C$ of $R$-modules is \emph{minimal} if $\partial^C_i(C_i)\subseteq \m C_{i-1}$ for all $i$.
%A complex is called \emph{contractible} if it is homotopically equivalent to the zero complex.
%We note that every totally acyclic complex $C$ may be decomposed as $C'\oplus Z$ where $C'$ is minimal and $Z$ is contractible.

\section{The `forgetful' triangle functor}\label{forgetful}

Let $Q$ be a commutative local Gorenstein ring, and $\varphi:Q\to R$ a surjective local ring homomorphism such that $\pd_QR<\infty$.

The main objective of this section is to define the `forgetful' functor
\[
T=T_\varphi:\Ktac(R) \to  \Ktac(Q)
\]
and prove that it is a triangle functor.  The definition of $T$ is as follows.

\begin{definition}
Let $C\in\Ktac(R)$.  Then $TC\in\Ktac(Q)$ is a complete resolution of $\Im\partial_0^C$ over $Q$ (which is uniquely defined by \ref{well-defined}.)  Given a morphism $[f]:C \to C'$ in $\Ktac(R)$, we have the map $\mu: \Im\partial_0^C \to\Im \partial_0^{C'}$ induced by the morphism $f:C\to C'$ of $R$-complexes.  Then $T[f]:TC\to TC'$ is the homotopy equivalence class $[\widehat\mu]$ of the comparison map
$\widehat\mu:TC\to TC'$ between complete resolutions (which is uniquely determined by $\mu$, by \ref{well-defined}.)
\end{definition}

\begin{proposition}  $T:\Ktac(R)\to\Ktac(Q)$ is an additive functor.
\end{proposition}

\begin{proof}  The only verifications of the axioms of additive functor that are possibly not obvious follow easily from the fact that $T$ is well-defined on morphisms in $\Ktac(R)$, which we now prove.  Let $[f]:C \to C'$ be the zero morphism in $\Ktac(R)$. Let $\mu:\Im\partial_0^C\to\Im\partial_0^{C'}$ be the map induced by $f$. Then there exists a homotopy
$\sigma\in\Hom_R(C,C')_1$
such that  $\overline{\partial_0^{C'}}(\sigma_{-1}|_{\Im\partial_0^C})=\mu$,  where $\overline{\partial_0^{C'}}:C_0'\to \Im\partial_0^{C'}$ is the map induced by $\partial_0^{C'}$. Let $F$ and $F'$ be $Q$-free resolutions of $\Im\partial_0^C$ and $\Im\partial_0^{C'}$, respectively, and $K$ a minimal $Q$-free resolution of $C_0'$.  Then
any chain map $\overline\mu:F\to F'$ lifting $\mu$ is homotopic to a lifting of the composition
\[
\xymatrixrowsep{2pc}
\xymatrixcolsep{3pc}
\xymatrix{
\cdots \ar@{->}[r] & F_1\ar@{->}[r]^{\partial_1^F}\ar@{.>}[d] & F_0 \ar@{.>}[d] \ar@{->}[r] & \Im\partial_0^C \ar@{->}[d]^{\sigma_{-1}|_{\Im\partial_0^C}} \ar@{->}[r] & 0\\
\cdots \ar@{->}[r] & K_1\ar@{->}[r]^{\partial_1^K}\ar@{.>}[d] & K_0 \ar@{.>}[d]\ar@{->}[r] & C_0' \ar@{->}[d]^{\overline{\partial_0^{C'}}}\ar@{->}[r] & 0\\
\cdots \ar@{->}[r] & F'_1 \ar@{->}[r]^{\partial_1^{F'}} & F'_0 \ar@{->}[r] & \Im\partial_0^{C'} \ar@{->}[r] & 0
}
\]
which is eventually zero since $K$ is a finite complex.  This shows that $\widehat\mu\sim 0$, in other words, $T[f]=0$.
\end{proof}

\begin{theorem} $T:\Ktac(R)\to\Ktac(Q)$ is a triangle functor.
\end{theorem}

\begin{proof} We first show that $T$ commutes with shifts, that is, there is a natural isomorphism between $T\Sigma$ and
$\Sigma T$. Let $M=\Im\partial_0^C$.
From Corollary \ref{trcone} below we see that a minimal free resolution of $\Omega^R_1(M)$ agrees with that of $M$ beginning at
degree $\pd_QR+1$.  Specifically if $F$ and $F'$ are minimal free resolutions of $M$ and $\Omega^R_1(M)$ over $Q$, respectively, then we have $\Sigma^{-1}(F_{\ge c+2}) \cong F'_{\ge c+1}$, for $c=\pd_QR$.  This implies that $\Sigma^{-1}(TC)\simeq T(\Sigma^{-1}C)$, which is another way of stating the result.  It is clear that this isomorphism respects morphisms in $\Ktac(R)$, that is, if $[f]:C \to C'$ is a morphism in $\Ktac(R)$, then the following diagram commutes.
\[
\xymatrixrowsep{2pc}
\xymatrixcolsep{3pc}
\xymatrix{
T\Sigma C \ar@{->}[r]^{T\Sigma[f]}\ar@{->}[d]^\simeq & T\Sigma C' \ar@{->}[d]^\simeq \\
\Sigma TC \ar@{->}[r]^{\Sigma T[f]} & \Sigma TC' 
}
\]

Next we show that $T$ takes distinguished triangles to distinguished triangles.  Any distinguished triangle in $\Ktac(R)$ is isomorphic as a triangle to one of the form $C\xrightarrow{[f]} C'\to \cone([f]) \to\Sigma C$.  We need to complete the diagram
\[
\xymatrixrowsep{2pc}
\xymatrixcolsep{3pc}
\xymatrix{
TC \ar@{->}[r]^{T[f]}\ar@{=}[d] & TC' \ar@{->}[r]\ar@{=}[d] & T\cone([f]) \ar@{->}[r]\ar@{.>}[d] & \Sigma(TC)\ar@{=}[d]\\
TC \ar@{->}[r]^{T[f]} & TC' \ar@{->}[r] & \cone(T[f]) \ar@{->}[r] & \Sigma(TC)
}
\]
so that the second two squares commute.

Let  $TC \xrightarrow{\rho} F \xrightarrow{\pi} C_{\ge -1}$ and $TC' \xrightarrow{\rho'} F' \xrightarrow{\pi'} C'_{\ge 0}$ be complete resolutions over $Q$, and
$\mu:C_{\ge -1} \to C'_{\ge 0}$ be the map induced by $f$.   Then by \ref{well-defined} there exists a morphism of $Q$-complexes $\overline\mu:F\to F'$ such that all squares in the following diagram commute (we can assume that $\pi'$ is surjective.)
\[
\xymatrixrowsep{1pc}
\xymatrixcolsep{2pc}
\xymatrix{
& F_2 \ar@{->}[rr]\ar@{->}[dd]^<<<<{\pi_2}\ar@{.>}[dl]^{\overline\mu_2} && F_1 \ar@{->}[rr]\ar@{->}[dd]^<<<<{\pi_1}\ar@{.>}[dl]^{\overline\mu_1}&& F_0 \ar@{->}[rr]\ar@{->}[dd]^<<<<{\pi_0}\ar@{.>}[dl]^{\overline\mu_0} && F_{-1}\ar@{->}[dd]^<<<<{\pi_{-1}}\\
F_2' \ar@{->}[rr]\ar@{->}[dd]^<<<<{\pi'_2} && F_1' \ar@{->}[rr]\ar@{->}[dd]^<<<<{\pi'_1} && F_0'\ar@{->}[dd]^<<<<{\pi'_0}  &&&\\
& C_2 \ar@{->}[rr]\ar@{->}[dl]^{\mu_2} && C_1 \ar@{->}[rr]\ar@{->}[dl]^{\mu_1}&& C_0 \ar@{->}[rr]\ar@{->}[dl]^{\mu_0} && C_{-1}\\
C_2' \ar@{->}[rr] && C_1' \ar@{->}[rr] && C_0' &&&\\
}
\]
This diagram gives rise to a commutative diagram of short exact sequences of morphisms of $Q$-complexes
\[
\xymatrixrowsep{2pc}
\xymatrixcolsep{3pc}
\xymatrix{
0 \ar@{->}[r]  & F' \ar@{->}[r] \ar@{->}[d]^{\pi'} & \cone(\overline\mu) \ar@{->}[r] \ar@{->}[d]^{\left(\begin{smallmatrix}\pi' & 0 \\ 0 & \Sigma\pi \end{smallmatrix}\right)} & \Sigma F \ar@{->}[r] \ar@{->}[d]^{\Sigma\pi} & 0\\
0 \ar@{->}[r]  & C' _{\ge 0} \ar@{->}[r]  & \cone(\mu) \ar@{->}[r] & \Sigma(C_{\ge -1}) \ar@{->}[r] & 0\\
}
\]
The resulting commutative diagram of long exact sequences of homology shows that the morphism of $Q$-complexes $\cone(\overline\mu) \to \cone(\mu)$ is a quasiisomorphism.  Thus $\cone(\overline\mu)$ is a semi-free resolution of $\cone(\mu)$, and therefore also of $\Im\partial_0^{\cone(f)}$.  It follows that
\[
\cone(\widehat\mu) \to \cone(\overline\mu) \to \Im\partial_0^{\cone(f)}
\]
is a complete resolution of $\Im\partial_0^{\cone(f)}$, where $\widehat\mu:TC \to TC'$ is a morphism of $Q$-complexes as in \ref{well-defined}.  We now have that the diagram of morphisms
\[
\xymatrixrowsep{2pc}
\xymatrixcolsep{3pc}
\xymatrix{
TC \ar@{->}[r]^{\widehat\mu}\ar@{=}[d] & TC' \ar@{->}[r]\ar@{=}[d] & T\cone(f) \ar@{->}[r]\ar@{=}[d] & \Sigma(TC)\ar@{=}[d]\\
TC \ar@{->}[r]^{\widehat\mu} & TC' \ar@{->}[r] & \cone(\widehat\mu) \ar@{->}[r] & \Sigma(TC)
}
\]
which commutes up to homotopy.  The result follows.
\end{proof}

For an arbitrary commutative local ring $A$, consider a short exact sequence of finitely generated $A$-modules $0\to X \to Y \xrightarrow{\pi} Z \to 0$.
Let $F$ and $G$ be minimal free resolutions of $Y$ and $Z$, respectively. Let $\cone(\phi)$ denote the mapping
cone of the morphism $\phi: F \to G$ lifting the surjection $\pi$:
\[
\cone(\phi): \cdots \to F_1 \oplus G_2 \xrightarrow{\left(\begin{smallmatrix} -\partial_1^F & 0 \\ \phi_1 & \partial_2^G \end{smallmatrix}\right)} F_0 \oplus G_1 \xrightarrow{\left(\begin{smallmatrix} \phi_0 & \partial_1^G \end{smallmatrix}\right)} G_0.
\]

\begin{lemma} In the notation of the previous discussion, the truncated mapping cone
\[
\trcone(\phi): \cdots \to F_2 \oplus G_3 \xrightarrow{\left(\begin{smallmatrix} -\partial_2^F & 0 \\ \phi_2 & \partial_3^G \end{smallmatrix}\right)} F_1 \oplus G_2 \xrightarrow{\left(\begin{smallmatrix} -\partial_1^F & 0 \\ \phi_1 & \partial_2^G \end{smallmatrix}\right)} \ker\left(\begin{matrix} \phi_0 & \partial_1^G \end{matrix}\right)
\]
is (after shift) a free resolution of $X$.
\end{lemma}

\begin{proof}
We first note that Since $F$ and $G$ are both minimal free resolutions, the map $\phi_0$ is surjective.  Therefore
so is the map $\left(\begin{smallmatrix} \phi_0 & \partial_1^G \end{smallmatrix}\right)$, and so
$\ker\left(\begin{smallmatrix} \phi_0 & \partial_1^G \end{smallmatrix}\right)$ is free.

From the short exact sequence of complexes $0\to G \to \cone(\phi) \to \Sigma F \to 0$, we get the long exact sequence of homology
\[
\cdots \to \H_1(G) \to \H_1(\cone(\phi)) \to \H_0(F) \xrightarrow{\pi} \H_0(G) \to \H_0(\cone(\phi)) \to 0.
\]
Since $\left(\begin{smallmatrix} \phi_0 & \partial_1^G \end{smallmatrix}\right)$ is surjective, $\H_0(\cone(\phi))=0$.
Thus this long exact sequence of homology reduces to the short exact sequence
\[
0 \to \H_1(\cone(\phi)) \to Y \xrightarrow{\pi} Z \to 0,
\]
and so $X\cong \H_1(\cone(\phi))$.  Since $\H_i(\cone(\phi))=\H_i(\trcone(\phi)$ for all $i\ge 1$, the result follows, that is $\Sigma^{-1}\trcone(\phi)$ is a free resolution of $X$.
\end{proof}

\begin{corollary}\label{trcone} Let $K$ be a minimal free resolution of $R$ over $Q$, $F$ a minimal resolution of $M$ over $Q$, and $\mu=\rank F_0$. Consider the morphism of complexes $\phi: K^\mu \to F$ lifting the surjection $F_0 \to M$.  Then $\Sigma^{-1}\trcone(\phi)$ is a free resolution of $\Omega^R_1(M)$ over $Q$. If $c=\pd_QR<\infty$ then $(\Sigma^{-1}\trcone\phi)_{\ge c+1}=\Sigma^{-1}(F_{\ge c+2})$
\end{corollary}

\section{Adjunction}

Keeping the same assumptions on $Q$ and $R$ as in the previous section, the goal of this section is to compare $\Ktac(R)$ with $\Ktac(Q)$.  This will be done by constructing an adjoint pair of triangle functors between the two categories.

The descension functor is easy:
\[
S=S_\varphi:\Ktac(Q)\to \Ktac(R)
\]
is defined by $SC=C\otimes_QR$ and $S[f]=[f\otimes_QR]$ for $C$ an object and $[f]$ a morphism in
$\Ktac(Q)$.  This is a triangle functor due in part to \ref{dual isomorphism}.  Indeed, if $C$ is acyclic complex of free $Q$-modules, then $C\otimes_QR$ is an acyclic complex of free $R$-modules (since $\pd_QR<\infty$), and $\Hom_Q(C,Q)$ being acyclic implies $\Hom_Q(C,Q)\otimes_QR$ is acyclic.  Hence 
$\Hom_R(C\otimes_QR,R)$ is acyclic by \ref{dual isomorphism}.  It is easy to see that $S$ takes homotopic morphisms of complexes to homotopic morphisms of complexes, commutes with shifts and takes distinguished triangles to distinguished triangles.

The ascension functor
\[
T=T_\varphi:\Ktac(R)\to\Ktac(Q)
\]
is the functor defined in Section \ref{forgetful}.

Our main result for this section is the following.

\begin{theorem}\label{adjunction}
The triangle functors $S$ and $T$ form an adjoint pair, that is, they satisfy the following property:
for all $C\in\Ktac(R)$ and $D\in\Ktac(Q)$ there exist a bijection
\begin{align}
\Hom_{\Ktac(Q)}(D,TC) \to \Hom_{\Ktac(R)}(SD,C)
\end{align}
which is natural in each variable.
\end{theorem}

Before engaging the proof, we observe that for $D\in \Ktac(Q)$ one has
\[
TSD\simeq D \otimes_Q K
\]
where $K$ is a $Q$-free resolution of $R$.  Indeed, one has that $D_{\ge 0}\otimes_Q K$ is a $Q$-free resolution of $\Im\partial_0^{SD}\cong \Im\partial_0^D\otimes_QR$.  The assertion is now clear.

\begin{proof} In order to prove the theorem we define natural transformations
\[
\eta:\Id_{\Ktac(Q)}\to TS
\]
and
\[
\epsilon:ST \to \Id_{\Ktac(R)}
\]
--- the unit and counit, respectively, of the adjunction --- as follows. For  $D\in\Ktac(Q)$ define $\eta_D : D \to TSD$ to be the morphism of complexes embedding $D_n$ into the first component
of $TSD_n=\bigoplus_{i=0}^n D_{n-i}\otimes_Q K_i$ for all $n$.  And for $C\in\Ktac(R)$ define $\epsilon_C : STC \to C$ to be the morphism induced by the comparison map $F \to C_{\ge 0}$, where $F$ is a free resolution of $\Im\partial_0^C$ over $Q$.
It follows from \ref{extendmorphism} and \ref{extendhomotopy} that $\eta$ and $\epsilon$ are natural in their arguments.

We just need to show that
\[
T\epsilon_C\circ\eta_{TC} \sim \Id_{TC} \text{\quad and \quad}  \epsilon_{SD}\circ S\eta_D \sim \Id_{SD}
\]
First we discuss the map $T\epsilon_C$.  We have a morphism of complexes
\[
\xymatrixrowsep{2pc}
\xymatrixcolsep{3pc}
\xymatrix{
\cdots \ar@{->}[r] & TC_1\ar@{->}[r]^{\partial_1^{TC}}\ar@{->}[d]^{\widehat\mu_1} & TC_0 \ar@{->}[d]^{\widehat\mu_0} \ar@{->}[r] & \Im\partial_0^{TC} \ar@{->}[d]^{\nu} \ar@{->}[r] & 0\\
\cdots \ar@{->}[r] & F_1\ar@{->}[r]^{\partial_1^F} & F_0\ar@{->}[r] & \Im\partial_0^C \ar@{->}[r] & 0
}
\]
where $F$ is a $Q$-free resolution of $\Im\partial_0^C$.
Suppose that $\ker\varphi=(x_1,\dots,x_r)$.  For each $1 \le i \le r$, let $\sigma_i:TC_0\to F_1$ be the beginning of  homotopies  expressing the fact that the morphisms $x_i\widehat\mu$ are null homotopic, that is,
$\partial_1^F\circ \sigma_i=x_i\widehat\mu_0$ for $1\le i \le r$. Now let $K$ be a minimal free resolution of $R$ over $Q$ and define maps
$u_0: TC_0\otimes_Q K_0 \to F_0$ by $u_0(a\otimes b)=\widehat\mu_0(a)b$, and
$u_1(a\otimes (b_1,\dots,b_r))=\sigma_1(a)b_1+\cdots +\sigma_c(a)b_r$ for
$a\otimes (b_1,\dots,b_r)\in TC_0\otimes_Q K_1$, $u_1(a\otimes b)=\widehat\mu_1(a)b$ for
$a\otimes b\in TC_1\otimes_Q K_0$.  Then one checks easily that the following diagram commutes.
\[
\xymatrixrowsep{2pc}
\xymatrixcolsep{4pc}
\xymatrix{
TC_0 \otimes_Q K_1 \ar@{}[d]<1.7 ex>_\bigoplus \ar@{->}[dr]^{TC_0\otimes \partial_1^K} &  \\
TC_1\otimes_Q K_0 \ar@{->}[r]^{\partial_1^{TC}\otimes K_0}\ar@{->}[d]^{u_1} & TC_0\otimes_Q K_0  \ar@{->}[d]^{u_0} \ar@{->}[r] & \Im\partial_0^{TC}\otimes_QR \ar@{->}[r] \ar@{->}[d]^{\nu\otimes R} & 0\\
F_1 \ar@{->}[r]^{\partial_1^F}& F_0  \ar@{->}[r] & \Im\partial_0^C \ar@{->}[r] & 0
}
\]
This gives rise to a morphism of $Q$-complexes $(TC_{\ge 0}\otimes_QK) \to F$ such that $u_n(a\otimes b)=\widehat\mu_n(a)b$ for $a\otimes b\in TC_n\otimes_Q K_0$, the former complex being a $Q$-free resolution of $\Im\partial_0^{TC}\otimes_QR$,
% by lifting the vertical maps on each direct summand of $TC \otimes K$.
It follows that we may achieve the morphism
$T\epsilon_C: TSTC \to TC$ satisfying $(T\epsilon_C)_n(a\otimes b) = ab$ for
$a\otimes b\in TC_n \otimes_Q K_0$.

We have the natural embedding $\eta_{TC}:TC \to TSTC$ with $\eta_{TC}(a)=a\otimes 1\in TC_n\otimes_Q K_0$ for all $a\in TC_n$.  Thus we have shown that $T\epsilon_C\circ\eta_{TC} \sim \Id_{TC}$.

The morphism $S\eta_D:SD\to STSD$ embeds $SD_n$ into the first component of
$STSD\simeq \bigoplus_{i=0}^cSD_{n-i}\otimes_QSK_i$ for all $n$.  And the morphism
$\epsilon_{SD}:STSD \to SD$ takes the first component of $STSD\simeq \bigoplus_{i=0}^cSD_{n-i}\otimes_QSK_i$ to $SD_n$ for all $n\in\mathbb Z$.  Thus we have
$\epsilon_{SD}\circ S\eta_D \sim \Id_{SD}$.
\end{proof}

\section{Approximations of totally acyclic complexes}

Our main application of Theorem \ref{adjunction} is a resulting notion of approximation in the homotopy category of totally acyclic complexes.  We now recall the notion of approximation we use, due to Auslander and Smal\o  \ \cite{AuslanderSmalo}, and independently,  Enochs \cite{Enochs}, 

Let $\mathcal X$ be a full subcategory of a category $\mathcal C$. Then a \emph{right $\mathcal X$-approximation} of
$C\in\mathcal C$ is a morphism $X\xrightarrow{\epsilon} C$, with $X\in\mathcal X$, such that for all objects $Y\in\mathcal X$,
the sequence $\Hom_{\mathcal C}(Y,X) \xrightarrow{\Hom(Y,\epsilon)} \Hom_{\mathcal C}(Y,C) \to 0$ is exact.  

Dually, one has the concept of left $\mathcal X$-approximations. Specifically, a morphism $C\xrightarrow{\mu} X$, with $X\in\mathcal X$, is called a \emph{left $\mathcal X$-approximation} of $C\in\mathcal C$ if for all objects $Y\in\mathcal X$,
the sequence $\Hom_{\mathcal C}(X,Y) \xrightarrow{\Hom(\mu,Y)} \Hom_{\mathcal C}(C,Y) \to 0$ is exact. 

The full subcategory $\mathcal X$ is called \emph{functorially finite} in $\mathcal C$ if for every object $C\in\mathcal C$,
there exists a right $\mathcal X$-approximation of $C$ and a left $\mathcal X$-approximation 
of $C$.

We let $S\Ktac(Q)=\{ D \otimes_QR \mid D\in\Ktac(Q)\}$.  Our main application of Theorem \ref{adjunction} is the following.

\begin{theorem}\label{approximate}
\qquad $S\Ktac(Q)$ is functorially finite in $\Ktac(R)$.
\end{theorem}

\begin{proof}
That every $C\in\Ktac(R)$ has a right $S\Ktac(Q)$-approximation follows immediately from Theorem \ref{adjunction}: the morphism $[\epsilon_C]:STC\to C$ is a right approximation in 
$\Ktac(R)$.  Indeed, if $[f]:SD\to C$ is any morphism in $\Ktac(R)$ with $D\in\Ktac(Q)$, then
from the natural transformation  $\epsilon:ST \to \Id_{\Ktac(R)}$ we have equality
$[\epsilon_C]\circ ST[f]=[f]\circ[\epsilon_{SD}]$.  Composing on the right with 
$S[\eta_D]$ we obtain $[\epsilon_C]\circ ST[f]\circ S[\eta_D]=[f]$, and thus
$ST[f]\circ S[\eta_D]:SD\to STC$ is the morphism we seek.

Now we show that every $C\in\Ktac(R)$ has a left approximation.  This can be done by simply dualizing a right approximation.  For this we will use several times that for any given $D\in\Ktac(Q)$, one has the natural isomorphism of complexes $\Hom_R(D\otimes_QR,R)\cong \Hom_Q(D,Q)\otimes_QR$ from \ref{dual isomorphism}. We have the right approximation $[\epsilon_{C^*}]: STC^* \to C^*$ of $C^*$.  The claim is that
$[\epsilon_{C^*}^*]:C\cong C^{**}\to (STC^*)^*$ is a left approximation of $C$.  Note that the target of $\epsilon_{C^*}^*$ is in $S\Ktac(Q)$ by the aforementioned isomorphism of \ref{dual isomorphism}.  Now let $E\in \Ktac(Q)$ and $f:C\to SE$ be a morphism in $\Ktac(R)$.   Then we have the morphism
$f^*:(SE)^*\to C^*$, with $(SE)^*$ in $S\Ktac(Q)$.  Therefore we have that $f^*\sim\epsilon_{C^*} g$ for some morphism $g:(SE)^*\to STC^*$.  Dualizing back we have that $f\sim g^*\epsilon_{C^*}^*$, which is what we needed to show.
\end{proof}

Motivated by results along the lines of \cite[Proposition 1.4]{Neeman}, we ask the following:

\begin{question}
Is $S\Ktac(Q)$ a thick subcategory of $\Ktac(R)$?
\end{question}
 
We illustrate Theorem \ref{approximate} with an example.

\begin{example} Let $R=k[x,y]/(x^2,y^2)$, and $C$ be the totally acyclic $R$-complex with 
$\Im\partial^C_0=Rxy\cong k$:
\[
C: \cdots \to R^3 \xrightarrow{\left(\begin{smallmatrix} x & 0 & -y \\ 0 & y & x \end{smallmatrix}\right)} R^2 \xrightarrow{\left(\begin{smallmatrix} x & y \end{smallmatrix}\right)} R \xrightarrow{\left(\begin{smallmatrix} xy \end{smallmatrix}\right)} R \xrightarrow{\left(\begin{smallmatrix} x \\ y \end{smallmatrix}\right)}R^2 \to\cdots
\]
Then a free resolution of $\Im\partial_0^C$ over $Q=k[x,y]/(x^2)$ is given by
\[
F: \cdots \to Q^2 \xrightarrow{\left(\begin{smallmatrix} x & -y \\ 0 & x \end{smallmatrix}\right)} Q^2 \xrightarrow{\left(\begin{smallmatrix} x & y \\ 0 & x \end{smallmatrix}\right)}Q^2 \xrightarrow{\left(\begin{smallmatrix} x & -y \\ 0 & x \end{smallmatrix}\right)}Q^2 \xrightarrow{\left(\begin{smallmatrix} x & y \end{smallmatrix}\right)} Q \to 0
\]
The right approximation $\epsilon_C:STC \to C$ takes the form
\[
\xymatrixrowsep{3pc}
\xymatrixcolsep{3pc}
\xymatrix{
\cdots \ar@{->}[r] & R^2\ar@{->}[r]^{\left(\begin{smallmatrix} x & -y \\ 0 & x \end{smallmatrix}\right)}\ar@{->}[d]^{\left(\begin{smallmatrix} 1 & 0 \\ 0 & 0 \\ 0 & 1 \end{smallmatrix}\right)} & R^2 \ar@{->}[d]^{\Id_{R^2}} \ar@{->}[r]^{\left(\begin{smallmatrix} x & y\\ 0 & x \end{smallmatrix}\right)} & R^2
\ar@{->}[d]^{\left(\begin{smallmatrix} 1 & 0 \end{smallmatrix}\right)} \ar@{->}[r]^{\left(\begin{smallmatrix} x & -y \\ 0 & x \end{smallmatrix}\right)} & R^2 \ar@{->}[d]^{\left(\begin{smallmatrix} y & 0 \end{smallmatrix}\right)} \ar@{->}[r]^{\left(\begin{smallmatrix} x & y\\ 0 & x \end{smallmatrix}\right)} & R^2 \ar@{->}[r]\ar@{->}[d]^{\left(\begin{smallmatrix} y & 0 \\ 0 & 0 \end{smallmatrix}\right)} & \cdots\\
\cdots \ar@{->}[r] & R^3\ar@{->}[r]_{\left(\begin{smallmatrix} x & 0 & -y \\ 0 & y & x \end{smallmatrix}\right)} & R^2 \ar@{->}[r]_{\left(\begin{smallmatrix} x & y \end{smallmatrix}\right)} & R \ar@{->}[r]_{\left(\begin{smallmatrix} xy \end{smallmatrix}\right)} & R \ar@{->}[r]_{\left(\begin{smallmatrix} x \\ y \end{smallmatrix}\right)} & R^2 \ar@{->}[r] & \cdots
}
\]

Since $C$ is self-dual in this example, that is $C\cong \Sigma^{-1} (C^*)$, the left approximation
$[\epsilon^*_C]:C\to (STC)^*$ takes the form
\[
\xymatrixrowsep{3pc}
\xymatrixcolsep{3pc}
\xymatrix{
\cdots \ar@{->}[r] & R^3\ar@{->}[r]^{\left(\begin{smallmatrix} x & 0 & -y \\ 0 & y & x \end{smallmatrix}\right)} & R^2 \ar@{->}[r]^{\left(\begin{smallmatrix} x & y \end{smallmatrix}\right)} & R \ar@{->}[r]^{\left(\begin{smallmatrix} xy \end{smallmatrix}\right)} & R \ar@{->}[r]^{\left(\begin{smallmatrix} x \\ y \end{smallmatrix}\right)} & R^2 \ar@{->}[r] & \cdots\\
\cdots \ar@{->}[r] & R^2\ar@{->}[r]_{\left(\begin{smallmatrix} x & 0\\ -y & x \end{smallmatrix}\right)}\ar@{<-}[u]_{\left(\begin{smallmatrix} y & 0 & 0\\ 0 & 0 & 0 \end{smallmatrix}\right)} & R^2 \ar@{<-}[u]_{\left(\begin{smallmatrix} y & 0 \\ 0 & 0 \end{smallmatrix}\right)} \ar@{->}[r]_{\left(\begin{smallmatrix} x & 0\\ y & x \end{smallmatrix}\right)} & R^2
\ar@{<-}[u]_{\left(\begin{smallmatrix} y \\ 0 \end{smallmatrix}\right)} \ar@{->}[r]_{\left(\begin{smallmatrix} x & 0 \\ -y & x \end{smallmatrix}\right)} & R^2 \ar@{<-}[u]_{\left(\begin{smallmatrix} 1 \\ 0 \end{smallmatrix}\right)} \ar@{->}[r]_{\left(\begin{smallmatrix} x & 0\\ y & x \end{smallmatrix}\right)} & R^2 \ar@{->}[r]\ar@{<-}[u]_{\Id_{R^2}} & \cdots
}
\]
\end{example}

Approximations may be trivial, in particular, when the projective dimension of 
$\Im\partial_0^C$ is finite over $Q$, as is the case in the next example.

\begin{example} Let $R=k[x,y]/(x^2,y^2)$ and $C$ the totally acyclic $R$-complex with $\Im\partial^C_0=Ry$:
\[
C: \cdots \to R \xrightarrow{\left(\begin{smallmatrix} y  \end{smallmatrix}\right)} R \xrightarrow{\left(\begin{smallmatrix}  y \end{smallmatrix}\right)} R \xrightarrow{\left(\begin{smallmatrix} y \end{smallmatrix}\right)} R \xrightarrow{\left(\begin{smallmatrix}  y \end{smallmatrix}\right)}R \to\cdots
\]
Then for $Q=k[x,y]/(x^2)$, $\pd_Q\Im\partial_0^C<\infty$ and the approximation is 
$[\epsilon_C]:0\to C$.
\end{example}

Recall (from \cite{AuslanderSmalo}, for example) that a morphism
$X\xrightarrow{\epsilon} C$ is called \emph{right minimal} if for every morphism $X\xrightarrow{f}X$ such that $\epsilon f = \epsilon$, we have that $f$ is an isomorphism.  We now point out that the right approximation $[\epsilon_C]:STC\to C$ may or may not be right minimal.

\begin{proposition}
Suppose that $D\in\Ktac(Q)$.  Then $[\epsilon_{SD}]:STSD\to SD$ is not a minimal approximation.
\end{proposition}

\begin{proof} Let $K$ be a $Q$-free resolution of $R$.  As described in the proof of \ref{adjunction}, $\epsilon_{SD}:STSD \to SD$ takes the first component of $STSD\simeq \bigoplus_{i=0}^cSD_{n-i}\otimes_QSK_i$ to $SD_n$ for all $n\in\mathbb Z$.  Thus taking as $[f]:STSD \to STSD$ the morphism sending $SD_n\otimes_QSK_0$ to itself and everything else to zero, we have 
$[\epsilon_{SD}]\circ[f]=[\epsilon_{SD}]$ and $[f]$ is not an isomorphism in $\Ktac(R)$.
\end{proof}

\begin{example} Let $R=k[x,y]/(x^2,y^2)$ and $C$ the totally acyclic complex
\[
C: \cdots \to R \xrightarrow{x}  R \xrightarrow{x} R \xrightarrow{x} R \to \cdots
\]
Then $M=\Im\partial_0^C=Rx$ and a free resolution of $M$ over $Q=k[x,y]/(x^2)$ is given by
\[
\cdots \to Q^2 \xrightarrow{\left(\begin{smallmatrix} x & -y^2 \\ 0 & x \end{smallmatrix}\right)} Q^2 \xrightarrow{\left(\begin{smallmatrix} x & y^2 \\ 0 & x \end{smallmatrix}\right)}Q^2 \xrightarrow{\left(\begin{smallmatrix} x & -y^2 \\ 0 & x \end{smallmatrix}\right)}Q^2 \xrightarrow{\left(\begin{smallmatrix} x & y^2 \end{smallmatrix}\right)} Q \to 0
\]
Thus $STC$ takes the form
\[
\cdots \to R^2 \xrightarrow{\left(\begin{smallmatrix} x & 0 \\ 0 & x \end{smallmatrix}\right)} R^2 \xrightarrow{\left(\begin{smallmatrix} x & 0 \\ 0 & x \end{smallmatrix}\right)}R^2 \xrightarrow{\left(\begin{smallmatrix} x & 0 \\ 0 & x \end{smallmatrix}\right)}R^2 \to\cdots
\]
and $\epsilon_C:STC \to C$ is given by $(\epsilon_C)_n=\left(\begin{smallmatrix} 1 & 0 \end{smallmatrix}\right)$ for all $n$.
This is not a minimal right approximation.  Indeed, consider the morphism $f: STC \to STC$ given by
$f_n=\left(\begin{smallmatrix} 1 & 0 \\ 0 & 0 \end{smallmatrix}\right)$. Then one has $\epsilon_Cf=\epsilon_C$ and $f$ is not a homotopy
equivalence.
\end{example}

\begin{chunk} \label{period2} {\bf Approximations by period 2 complexes.}
Recall that a local ring $Q$ is a \emph{hypersurface ring} if $Q$ is the quotient of a regular local ring by a principal ideal; hypersurface rings are Gorenstein.  In this case, Eisenbud \cite{Eisenbud} has shown that totally acyclic complexes are always periodic of period at most two. Thus in the setup where $Q$ is a hypersurface (and, as always, that $\pd_QR<\infty$ via $\varphi:Q\to R$), our
approximations compare nonperiodic totally acyclic complexes with those of period two.  This
setup occurs when $R$ has an embedded deformation \cite{Avramov}.
\end{chunk}

We next state a few results for later reference.  They have to do with compositions of approximations, in two different senses.

\begin{proposition}
Consider a sequence of finite local ring homomorphisms 
\[
Q \xra{\varphi} R' \xra{\psi} R
\] 
such that $Q$ and $R'$ are Gorenstein,  $\pd_QR'<\infty$, and $\pd_{R'}R<\infty$.  Then
$S_{\psi\varphi}$ and $T_{\psi\varphi}$ are naturally isomorphic to $S_{\psi}S_{\varphi}$ and $T_{\varphi}T_{\psi}$, respectively.
\end{proposition}

\begin{proof}
This follows from the fact that the assertion is clear for the $S$ functors, and from
uniqueness of adjoints. 
\end{proof}

%\begin{theorem}
%Consider the following diagram of finite local homomorphisms of Gorenstein rings, such
%that $\pd_Q R_1, \pd_Q R_2, \pd_{R_1} R$ and $\pd_{R_2} R$ are all finite
%\begin{equation*}
%\xymatrix{
  %& Q \ar[dr]^{\varphi'} \ar[dl]_{\varphi} & \\
%R_1  \ar[dr]_{\psi} & & R_2 \ar[dl]^{\psi'} \\
  %& R  &
%}
%\end{equation*}
%If $\Tor^Q_i(R_1,R_2) = 0$ for all $i>0$, then 
%$$
%\epsilon_{\varphi\psi}(C) = \epsilon_{\varphi'\psi'}(C) \cong \epsilon_{\psi'}(\epsilon_{\psi}(C)).
%$$
%\end{theorem}

%\begin{proof}
%THIS DEFINITELY NEEDS A PROOF-elevated from an observation to a theorem.
%\end{proof}

%\begin{definition}
%A discussion/definition of the lifted cone, and what $TC$ is when $\pd_Q R = 1$ should be given %here.  Together with the
%fact that the $T$ functors compose, one can iterate this procedure to compute what $TC$ is for any
%quotient of $Q$ by a regular sequence.
%\end{definition}

%\begin{observation}
%It may be possible to handle the case where $R$ is a quotient of $Q$ by a regular sequence all at once using the following
%trick:  Suppose that $R = Q/(f_1,\dots,f_c)$ with $(f_1,\dots,f_c)$ a regular sequence, and let $t_1,\dots,t_c : \Sigma^{-2}C \to C$
%be the cohomology operators.  Since the $t_1,\dots,t_c$ commute in $\Ktac(R)$, we have a Koszul complex of cohomology operators
%$$\Sigma^{-2c} C \xra{\begin{pmatrix}(-1)^{c-1}t_c \\ \vdots \\ t_1 \end{pmatrix}} (\Sigma^{-2c+2}C)^c
%\to \cdots \to (\Sigma^{-2} C)^c \xra{\begin{pmatrix}t_1 & t_2 & \cdots & t_c\end{pmatrix}} C.$$
%\end{observation}

\begin{chunk} \label{resolution}
Upon computing the right approximation $[\epsilon_C] : STC \to C$, one may iterate this process.  Indeed, complete
$[\epsilon_C]$ to a triangle in $\Ktac(R)$ and rotate it to obtain
$$\Sigma^{-1} \cone([\epsilon_C]) \to STC \to C \to.$$
Now compute a right approximation of $\Sigma^{-1} STC$, and repeat.  One then obtains a sequence of maps
in $\Ktac(R)$:
$$\mathbf{B} : \cdots \to B_3 \to B_2 \to B_1 \to B_0 \to C$$
where $B_0$ is a right approximation of $C$, $B_1$ is a right approximation of $\Sigma^{-1} \cone(B_0 \to C)$, etc.
Note that since composing two consecutive maps in an exact triangle is the zero map, one has
that the same holds for the maps in $\mathbf{B}$.
\end{chunk}

\begin{proposition}
Let $\varphi : Q \to R$ be a surjective local homomorphism of Gorenstein rings with $\pd_Q R = 1$, and let $C \in \Ktac(R)$.
If $[\epsilon_C] : STC \to C$ is the right approximation of $C$, then $\cone([\epsilon_C])$ is isomorphic to $\Sigma^2 C$.
\end{proposition}

\begin{proof}
By the assumptions of the proposition, we have that  $R = Q/(x)$ for some nonzerodivisor $x$ contained in the maximal ideal of $Q$.  Now let $t : C \to \Sigma^2 C$ be the cohomology operator
defined by the embedded deformation $R = Q/(x)$ (see \cite{Avramov}). Then one can show that we have an exact triangle
$$C \xra{t} \Sigma^2 C \to \Sigma STC \to$$
where the last arrow is the approximation map $\Sigma[\epsilon_C]$.  Therefore, after rotation we have
$$STC \xra{[\epsilon_C]} C \xra{t} \Sigma^2 C \to$$
which proves the claim.
\end{proof}

The following proposition follows from the previous one.

\begin{proposition}
Let $\varphi : Q \to R$ be a surjective local homomorphism of Gorenstein rings with $\pd_Q R = 1$, and let $C \in \Ktac(R)$.
Then the triangle resolution, as in \ref{resolution}, of $C$ in $\Ktac(R)$ with respect to $\Ktac(Q)$ has the form
$$\cdots \to \Sigma^{-2}STC \to \Sigma^{-1}STC \to STC \to C.$$
\end{proposition}

%\begin{observation}
%Let $\varphi : Q \to R$ be a homomorphism of Gorenstein rings with $Q$ a hypersurface, $\pd_Q R < \infty$,
%and let $C \in \Ktac(R)$.  Let $b = \sum_{i > 0} \rank_Q F_i$ where $\mathbf{F} \to R$ is a minimal free
%resolution of $R$ over $Q$.

%Then the ranks of the free modules in the $j^\text{th}$ stage of resolution of $C$ in $\Ktac(R)$ is $(\rank STC_0)b^{j-1}$.
%\end{observation}


\begin{thebibliography}{BeJoOp}
%\bibitem[AuBr]{AuslanderBridger}M.\ Auslander,  M.\ Bridger,
\bibitem[AuSm]{AuslanderSmalo} M.\ Auslander and S.\ O. \ Smal\o,  {\it Preprojective modules over Artin algebras\/} J. Algebra {\bf 66}(1) (1980), 61Ð122. 
\bibitem[Av]{Avramov} L.~Avramov, \emph{Homological asymptotics of modules over local rings}, Commutative Algebra (Berkeley, 1987), MSRI Publ. 15, Springer, New York 1989; pp. 33--62.
\bibitem[AvMa]{AvramovMartsinkovsky} L.~Avramov and A.~Martsinkovsky, \emph{Absolute, relative, and Tate cohomology of modules of finite Gorenstein dimension}  Proc. London Math. Soc. (3) {\bf 85} (2002), 393--440.
\bibitem[BeJoOp]{BerghJorgensenOppermann} P. A.\ Bergh, D. A.\ Jorgensen and S.\ Oppermann, {\it The Gorenstein  defect category\/} Quarterly Journal of Mathematics {\bf 6} (2015), no. 2, 459--471.
%\bibitem[BruHe]{BrunsHerzog}W.\ Brunz,  J.\ Herzog,
\bibitem[Bu]{Buchweitz}R.-O.\ Buchweitz, \emph{Maximal Cohen-Macaulay modules and Tate-cohomology over Gorenstein rings}, 1987, 155 pp; see \\https://tspace.library.utoronto.ca/handle/1807/16682.
\bibitem[ChFoHo]{CFH} L.~W.~Christensen, H.-B.~Foxby and H.~Holm, \emph{Derived category methods in commutative algebra} in preparation.
\bibitem[Ei]{Eisenbud} D.~Eisenbud, \emph{Homological algebra on a complete intersection, with an application to group representations}, Trans. Amer. Math. Soc. {\bf 260} (1980), no. 1, 35--64.
\bibitem[En]{Enochs} E.\ E.\ Enochs, \emph{Injective and ßat covers, envelopes and resolvents}, Israel J. Math. {\bf 39} (1981), no. 3, 189--209.
\bibitem[Kr]{Krause} H.\ Krause, {\it Approximations and adjoints in homotopy categories\/}, Math. Ann. {\bf 353} (2012), no. 3, 765--781. 
\bibitem[Ne]{Neeman}  A.~Neeman, \emph{Some adjoints in homotopy categories}, Ann.Math (2) {\bf 171}(3) (2010), 2143--2155.
\end{thebibliography}
\end{document}